\documentclass{conm-p-l}

\usepackage{amsthm,amsfonts, cite}
\usepackage{amssymb,graphicx,color, tikz-cd, tikz,float}
\usepackage[all]{xy}
\usepackage{eucal}
\usepackage{enumerate}
\usepackage{hyperref}
\numberwithin{equation}{section}

\theoremstyle{plain}
\newtheorem{teo}{Theorem}[section]
\newtheorem*{teo*}{Theorem}
\newtheorem{teoA}{Theorem}

\newtheorem*{cor*}{Corollary}
\newtheorem{lem}[teo]{Lemma}
\newtheorem*{lem*}{Lemma}
\newtheorem{prop}[teo]{Proposition}
\newtheorem*{prop*}{Proposition}

\theoremstyle{remark}
\newtheorem{obs}[teo]{Remark}

\newcommand{\R}{\ensuremath{{\mathbb{R}}}}

\renewcommand\ker{\mathrm{Ker}\,}

\makeatletter
\@namedef{subjclassname@2020}{\textup{2020} Mathematics Subject Classification}
\makeatother

\makeatletter
\newcommand*\KN{\mathpalette\@KN\relax}
\newcommand*\@KN[1]{%
  \mathbin{%
    \ooalign{%
      $#1\m@th\bigcirc$\cr
      \hidewidth$#1\m@th\wedge$\hidewidth\cr
    }%
  }%
}
\makeatother

\newcommand{\ovee}{\mathbin{\vcenter{\hbox{\ooalign{$\vee$\cr\hidewidth$\bigcirc$\hidewidth\cr}}}}}

\makeatletter
\newcommand{\cyclicsum}{\mathop{\mathpalette\cyclic@sum\relax}\limits}
\newcommand{\cyclic@sum}[2]{%
  \vcenter{\hbox{%
    \ifx#1\displaystyle\def\f{1.4}\else\def\f{1.0}\fi
    \begin{tikzpicture}[scale=\f, x=1em, y=1em, baseline=-0.1ex, line cap=round, line join=round]
      \tikzset{every path/.style={line width=\f*0.065em}}
      
      \draw (-0.35,0) -- (0.35,0);
      \draw (0,-0.35) -- (0,0.35);
      
      \draw (0,0.35) arc (90:-250:0.35);
      
      \fill[shift={(0,0.35)}, scale=\f] (0,0) -- (-0.11,0.06) -- (-0.07,0) -- (-0.11,-0.06) -- cycle;
    \end{tikzpicture}%
  }}%
}
\makeatother

\begin{document}

\title[Fedosov manifolds with parallel symplectic Weyl curvature]{Fedosov manifolds with parallel symplectic Weyl curvature}

\author[S. Rankin]{Shane Rankin}
\author[I. Terek]{Ivo Terek}

\address{\parbox{\linewidth}{Department of Mathematics \\ University of California, Riverside \\ Riverside, CA 92521, USA \\[-.5em] }}

\email{shane.rankin@email.ucr.edu}
\email{ivo.terek@ucr.edu}

\makeatletter
\@namedef{subjclassname@2020}{\textup{2020} Mathematics Subject Classification}
\makeatother

\keywords{symplectic connections $\cdot$ Fedosov manifolds $\cdot$ parallel Weyl tensor}
\subjclass[2020]{Primary: 53D05 Secondary: 53C05}

\begin{abstract}
A classical 1977 result by Roter states that a Riemannian manifold of dimension at least four with parallel Weyl curvature must be locally symmetric or conformally flat. We establish Roter's theorem in the symplectic category: namely, a Fedosov manifold with parallel symplectic Weyl curvature must be locally symmetric or Weyl-flat.
\end{abstract}

\maketitle

\section{Introduction}

By a \emph{Fedosov manifold} we mean, as usual, a triple $(M,\omega,\nabla)$ where $(M,\omega)$ is a connected symplectic manifold and $\nabla$ is a torsionfree connection on $M$ for which $\nabla\omega=0$ \cite{GRS_1998}; the connection is explicitly taken to be part of such geometric data because, unlike in the pseudo-Riemannian setting, it is not uniquely determined by $\omega$. Such connections, also called \emph{symplectic}, have played a central role in Fedosov's seminal work on formal deformation quantization of symplectic manifolds \cite{Fedosov_1994} and received considerable attention since then.

\medskip

The space of all four-times covariant tensors on a symplectic vector space $(V,\Omega)$ formally satisfying the symmetries of the curvature tensor $R$ of a symplectic connection admits a natural direct-sum decomposition into ${\rm Sp}(V,\Omega)$-irreducible factors, cf. \cite{Vaisman_1985}---some details, needed for later, are recalled in Section \ref{sec:inv-decomp}. This means that, for a Fedosov manifold as above, one may write $R = E+W$, where the \emph{symplectic Einstein tensor} $E$ is algebraically built from $\omega$ and the Ricci tensor of $\nabla$, while the \emph{symplectic Weyl tensor} $W$ is fully traceless.

\medskip

The zeroth-order conditions $E=0$ (equivalent to ${\rm Ric}=0$) and $W=0$ have been studied and are rather well-understood. There are many examples of Ricci-flat manifolds which are not Weyl-flat \cite[Theorem 4.1]{CG_2007}, as well as examples of Weyl-flat ones which are not Ricci-flat \cite[Corollary 9]{BC_2001}. The symplectic Weyl tensor is the only obstruction to integrability of a certain almost-complex structure built from $\nabla$ in the bundle of such $\omega$-compatible structures on $M$ \cite[Theorem 6.5]{BCGRS_2006}, while local models for real-analytic Weyl-flat manifolds are also known \cite[Theorem 4.3]{CGS_2005}. The first-order condition $\nabla {\rm Ric} = 0$ has also been investigated, and the local product structure of such manifolds has been determined, cf. \cite[Proposition 10]{CGR_1998} and \cite[Section 3]{Boubel_2003}.

\medskip

The remaining first-order condition, $\nabla W=0$, however, does not seem to have received any meaningful attention. This is not surprising from a Riemannian perspective: by a 1977 theorem due to Roter, a Riemannian manifold of dimension at least four having parallel Weyl curvature must be locally symmetric or conformally flat \cite[Theorem 2]{TensorNS77}. Thus, one may naturally wonder if Roter's theorem holds for Fedosov manifolds:
\begin{equation}\label{quest:roter}
  \parbox{.6\textwidth}{If a Fedosov manifold has parallel symplectic Weyl ten\-sor, must it be locally symmetric or Weyl-flat?}
\end{equation}
While \eqref{quest:roter} trivially holds for surfaces, with $W=0$ for dimensional reasons, its answer in general is not obvious.

\medskip On one hand, Roter's original proof via classical tensor calculus methods uses positive-definiteness of the metric in a crucial manner to ultimately conclude from $\|\nabla {\rm Ric}\|^2=0$ that $\nabla {\rm Ric} = 0$, and the result is outright false for indefinite metrics. The so-called \emph{essentially conformally symmetric (ECS)} manifolds---that is, those with $\nabla W=0$ without $\nabla R=0$ or $W=0$---have been largely investigated by Roter, Derdzinski, and the second author; see for instance \cite{Tohoku07, DR_2009, DT_2, AGAG23}, or \cite{Terek-survey} for a survey of more recent results. Without any positivity properties assumed for $\omega$, this suggests that the answer to \eqref{quest:roter} might be negative.

\medskip

On the other hand, from further results by Cahen--Gutt--Rawnsley, it turns out the answer to \eqref{quest:roter} is in fact positive when $\nabla$ is the Levi-Civita connection of a pseudo-K\"{a}hler metric having $\omega$ as its K\"{a}hler form, even when such pseudo-K\"{a}hler metric has indefinite signature. Indeed, recalling that a symplectic connection is called \emph{preferred} when it satisfies the gauge condition
  \begin{equation}
    (\nabla_X{\rm Ric})(Y,Z) +  (\nabla_Y{\rm Ric})(Z,X) +    (\nabla_Z{\rm Ric})(X,Y) =0
  \end{equation}for all vector fields $X,Y,Z$, cf. \cite[Section 3]{Bourgeois-Cahen_1999}, we have that $\nabla$ is preferred whenever $\nabla W = 0$ \cite[Remark 1]{CGR_1998}, while in the pseudo-K\"{a}hler case $\nabla$ is preferred if and only if $\nabla {\rm Ric}=0$ \cite[Proposition 2]{CGR_1998}.

  \medskip

  So, which is it? We settle the debate: the answer to \eqref{quest:roter} is \emph{always} positive.

\begin{teoA}\label{teo:symplectic-Roter} 
  A Fedosov manifold with parallel symplectic Weyl tensor must necessarily be locally symmetric or Weyl-flat.
\end{teoA}

In other words, there are no ``symplectic-ECS manifolds.'' The proof strategy we use to establish Theorem \ref{teo:symplectic-Roter} differs from the one adopted by Roter, and instead draws from the representation theory of the real symplectic Lie algebra.

\medskip

Here is a brief outline of the argument. First, the differential Bianchi identity combined with the condition $\nabla W = 0$ yields the existence of a vector field $U$ on $M$ such that $\nabla_Z{\rm Ric}$ is always proportional to $\underline{U} \odot \underline{Z}$ (Proposition \ref{prop:DRic-U}). Then, further differentiating the condition that $R(X,Y)\cdot W = 0$ for all $X,Y$ allows us to identify a certain subspace $\mathcal{S}$ of $\mathfrak{sp}_{2n}(\R)$ such that $\mathcal{S} \cdot W = 0$, cf. \eqref{eqn:subspace-S}. Finally, when $U \neq 0$, we use that $\mathfrak{sp}_{2n}(\R)$ is simple and show that the subalgebra of $\mathfrak{sp}_{2n}(\R)$ generated by $\mathcal{S}$ is in fact all of $\mathfrak{sp}_{2n}(\R)$ (Proposition \ref{prop:S-generates-spV}). The conclusion follows as then $\mathfrak{sp}_{2n}(\R)\cdot W=0$ can only happen when $W=0$ (Proposition \ref{prop:killed-by-derivations}).

\medskip

The method above can conceivably be used to give an alternative proof of Roter's original theorem, except perhaps in dimension four as $\mathfrak{so}(4)$ is not simple, and also illustrates its failure for indefinite metrics---see Remarks \ref{obs:failure-so(p,q)} and \ref{obs:p-maximal-Riem} for how the subtler parts of the argument work in the Riemannian setting.

\medskip

\noindent {\bf Acknowledgements:} We would like to thank Wee Liang Gan for pointing to us reference \cite{Bourbaki_1975}, used in Proposition \ref{prop:S-generates-spV}. 

\section{Symplectic curvature tensors and Hamiltonian derivations}\label{sec:symp-curvature}

Throughout this section, let $(V,\Omega)$ be a $2n$-dimensional symplectic vector space. For every vector $x\in V$, we let $\underline{x} = \Omega(x,\cdot)$ be its image under the isomorphism $V\cong V^*$ induced by $\Omega$.

\subsection{The invariant decomposition.}\label{sec:inv-decomp}  By a \emph{symplectic curvature tensor} on $V$ we mean a four-times covariant tensor $K$ on $V$ satisfying
\begin{equation}\label{eqn:defn-symp-curv}\parbox{.9\textwidth}{
  \begin{enumerate}[(i)]
  \item $K(y,x,z,w) = - K(x,y,z,w)$, 
  \item $K(x,y,w,z) = K(x,y,z,w)$,
  \item $K(x,y,z,w)+K(y,z,x,w) + K(z,x,y,w) = 0$,
  \end{enumerate}}
\end{equation}for all $x,y,z,w\in V$, and we let $\mathcal{R}(V,\Omega)$ denote the vector space of all such symplectic curvature tensors. The \emph{symplectic Kulkarni-Nomizu product} of $F\in \scalebox{1.2}{$\wedge$}^2V^*$ and $G\in S^2 V^*$ is defined by
\begin{equation}\label{eqn:SKN}
  \begin{split}
    (F \ovee G)(x,y,z,w) &= G(y,z)F(x,w) - G(x,z)F(y,w) \\ &\quad +F(x,z)G(y,w)  - F(y,z)G(x,w)+2F(x,y)G(z,w),
  \end{split}
\end{equation}and is easily verified to satisfy \eqref{eqn:defn-symp-curv}. The \emph{Ricci contraction} of $K\in \mathcal{R}(V,\Omega)$ is defined as ${\rm Ric}(K)(y,z) = -{\rm tr}_\Omega((x,w) \mapsto K(x,y,z,w))$, where ${\rm tr}_\Omega\colon [V^*]^{\otimes 2}\to \R$ is described with the aid of a fixed basis for $V$ as ${\rm tr}_\Omega(B) = \Omega^{ij}B_{ij}$; here, $[\Omega^{ij}]$ is the inverse of a matrix representation $[\Omega_{ij}]$ of $\Omega$, with the convention that $\Omega^{ij}\Omega_{jk} = \delta^i_k$. A routine computation using \eqref{eqn:SKN} yields ${\rm Ric}(\Omega\ovee G) = 2(n+1)G$, implying that
\begin{equation}\label{eqn:defn_Weyl}
  \parbox{.8\textwidth}{the \emph{Weyl component} of $K$, given by $\displaystyle{W(K) = K -  \frac{\Omega\ovee {\rm Ric}(K)}{2(n+1)}}$,}
\end{equation}has vanishing Ricci contraction, and leading to the ${\rm Sp}(V,\Omega)$-invariant direct-sum decomposition $\mathcal{R}(V,\Omega) = {\rm Im}(\Omega\ovee \cdot) \oplus \ker({\rm Ric})$.

\subsection{Derivations and basic properties} For each integer $k\geq 1$, there is a natural representation of $\mathfrak{sp}(V,\Omega)$ on the space of covariant $k$-tensors on $V$, via derivations: namely, for any $A\in \mathfrak{sp}(V,\Omega)$ and $\Theta\in [V^*]^{\otimes k}$, we consider the tensor $A\cdot \Theta \in [V^*]^{\otimes k}$ defined by
\begin{equation}\label{eqn:defn-derivation}
  (A\cdot \Theta)(x_1,\ldots, x_k) = -\sum_{i=1}^k \Theta(x_1,\ldots, Ax_i,\ldots, x_k),
\end{equation}for all $x_1,\ldots, x_k\in V$. Some basic properties of this operation are
\begin{equation}\label{eqn:basic-derivations}
  \parbox{.9\textwidth}{
    \begin{enumerate}[(i)]
    \item $A\cdot \Omega=0$,
    \item $A\cdot(B\cdot \Theta) - B\cdot(A\cdot \Theta) = [A,B]\cdot \Theta$,
    \item $A\cdot (F\ovee G) = (A\cdot F)\ovee G + F \ovee (A\cdot G)$,
    \end{enumerate}
}
\end{equation}for all $A,B\in\mathfrak{sp}(V,\Omega)$, $\Theta \in [V^*]^{\otimes k}$, $F\in \scalebox{1.2}{$\wedge$}^2V^*$ and $G\in S^2 V^*$.

\medskip

Given $K\in [V^*]^{\otimes 4}$ and $x,y\in V$, we consider the endomorphism $\bar{K}(x,y)$ of $V$ characterized by requiring that $\Omega(\bar{K}(x,y)z,w) = K(x,y,z,w)$ for all $z,w\in V$. Note that $\bar{K}(x,y) \in \mathfrak{sp}(V,\Omega)$ whenever $K$ satisfies \mbox{(\ref{eqn:defn-symp-curv}-ii)}. The following auxiliary lemma will be used in Section \ref{sec:subspaces-kill-W}.

\begin{lem}\label{lem:commutator[A,K]}
  Whenever $A\in \mathfrak{sp}(V,\Omega)$ and $K\in [V^*]^{\otimes 4}$, it holds that
  \begin{equation}\label{eqn:commutator[A,K]}
    [A, \bar{K}(x,y)] = \overline{A\cdot K}(x,y) + \bar{K}(Ax,y)+\bar{K}(x,Ay)
  \end{equation}for all $x,y\in V$.
\end{lem}

\begin{proof}
The relation  
\begin{equation}\label{eqn:A-dot-K}
  \begin{split}
    (A\cdot K)(x,y,z,w) =  &-K(Ax,y,z,w) - K(x,Ay,z,w) \\ &-K(x,y,Az,w)-K(x,y,z,Aw),
  \end{split}    
\end{equation}immediate from \eqref{eqn:defn-derivation}, may be rewritten as
\begin{equation}\label{eqn:A-dot-K-2}
  \begin{split}
    \Omega( \overline{A\cdot K}(x,y)z,w) = &-\Omega(\bar{K}(Ax,y)z,w)-\Omega(\bar{K}(x,Ay)z,w) \\ &+ \Omega([A,\bar{K}(x,y)]z,w)
  \end{split}
\end{equation}by using that $A\in\mathfrak{sp}(V,\Omega)$ to combine the last two terms in the right side of \eqref{eqn:A-dot-K} into the term involving $[A,\bar{K}(x,y)]$. Eliminating $z,w$ from \eqref{eqn:A-dot-K-2} and rearranging it, we obtain \eqref{eqn:commutator[A,K]}.
\end{proof}

Finally, we show that the representation \eqref{eqn:defn-derivation} restricted to $\mathcal{R}(V,\Omega)$ has no nontrivial invariants.

\begin{prop}\label{prop:killed-by-derivations}
  If $K\in \mathcal{R}(V,\Omega)$ has $A\cdot K = 0$ for every $A\in \mathfrak{sp}(V,\Omega)$, then necessarily $K=0$.
\end{prop}

\begin{proof}
  Due to \mbox{(\ref{eqn:defn-symp-curv}-ii)}, it suffices to show that $K(x,y,z,z) = 0$ for all vectors $x,y,z\in V$. We achieve this by establishing that
  \begin{equation}\label{eqn:Kpm2K}
    {\rm (i)}~K(x,y,z,z) - 2K(y,z,x,z) = 0,\quad {\rm (ii)}~K(x,y,z,z) + 2K(y,z,x,z) = 0,
  \end{equation}instead. With $\underline{x}\otimes x \in \mathfrak{sp}(V,\Omega)$, we may use \eqref{eqn:defn-derivation} together with \mbox{(\ref{eqn:defn-symp-curv}-i)} to directly compute that
  \begin{equation}\label{eqn:Om(x,.)x.K}
    0 = -((\underline{x}\otimes x) \cdot K)(x,y,y,y) = 2\Omega(x,y)K(x,y,x,y).
  \end{equation}As the set $\{(x,y) \in V\times V : \Omega(x,y) \neq 0\}$ is dense in $V\times V$ due to nondegeneracy of $\Omega$, continuity of $K$ leads via \eqref{eqn:Om(x,.)x.K} to $K(x,y,x,y) = 0$ for all $x,y\in V$. Polarizing this last relation in the variable $x$ we obtain
  \begin{equation}\label{eqn:first-polarization}
    K(x,y,z,y) + K(z,y,x,y) = 0,\quad\mbox{ for all }x,y,z\in V,
  \end{equation}and further polarizing \eqref{eqn:first-polarization} in the variable $y$ yields
  \begin{equation}\label{eqn:second-polarization}
    K(x,y,z,z)+K(x,z,z,y)+K(z,y,x,z) + K(z,z,x,y) = 0.
  \end{equation}
Applying \mbox{(\ref{eqn:defn-symp-curv}-i)} to the third and fourth terms, \mbox{(\ref{eqn:defn-symp-curv}-ii)} to the second term, and then \eqref{eqn:first-polarization} with the roles of $y$ and $z$ switched, \eqref{eqn:second-polarization} reduces to \mbox{(\ref{eqn:Kpm2K}-i)}. Obtaining \mbox{(\ref{eqn:Kpm2K}-ii)} as well is now straightforward: apply \mbox{(\ref{eqn:defn-symp-curv}-i)} and then \eqref{eqn:first-polarization} with the roles of $y$ and $z$ switched to the third term of \mbox{(\ref{eqn:defn-symp-curv}-iii)} with $w=z$.
\end{proof}

\begin{obs}\label{obs:failure-so(p,q)}
  Proposition \ref{prop:killed-by-derivations} does not hold as stated in the pseudo-Riemannian setting. Namely, whenever $(V,g)$ is a pseudo-Euclidean vector space and a curvaturelike tensor $R$ satisfies that $A\cdot R = 0$ for all $A\in \mathfrak{so}(V,g)$, then $R$ must be a multiple of $R_0$ given by $R_0(x,y,z,w) = g(y,z)g(x,w)-g(x,z)g(y,z)$. Indeed, in this case every ${\rm SO}(V,g)$-invariant $4$-tensor $R$ must be a linear combination of $g(x,y)g(z,w)$, $g(x,z)g(y,w)$, and $g(x,w)g(y,z)$, cf. \cite[Theorem 5.3.3]{Goodman-Wallach}, and imposing \mbox{(\ref{eqn:defn-symp-curv}-i)}, skew-symmetry in the last two arguments, and \mbox{(\ref{eqn:defn-symp-curv}-iii)}, leads to the claimed conclusion. In particular, this implies that if $W$ is a \emph{Weyl} curvaturelike tensor with $A\cdot W = 0$ for all $A\in \mathfrak{so}(V,g)$, then $W=0$. It is also part of \cite[Theorem 5.3.3]{Goodman-Wallach} that any ${\rm Sp}(V,\Omega)$-invariant $4$-tensor is a linear combination of $\Omega(x,y)\Omega(z,w)$, $\Omega(x,z)\Omega(y,w)$, and $\Omega(x,w)\Omega(y,z)$, making the real content of Proposition \ref{prop:killed-by-derivations} that no such nontrivial combinations survive \eqref{eqn:defn-symp-curv}.
\end{obs}

\subsection{Subspaces ultimately annihilating $W$}\label{sec:subspaces-kill-W}

A substantial part of the proof of Theorem \ref{teo:symplectic-Roter} may be carried out in the general setting of symplectic vector spaces. We thus present it here first, and then invoke it in Section \ref{sec:proof-teoA} at the right moment. Observe that the mapping
\begin{equation}\label{eqn:isomorphism-spv}
\parbox{.82\textwidth}{$\mathfrak{sp}(V,\Omega) \ni A \mapsto \Omega(\cdot, A\cdot) \in S^2V^*$ is an isomorphism of vector spaces,}
\end{equation}being injective due to nondegeneracy of $\Omega$, and between spaces of the same dimension $n(2n+1)$. Then, given a vector $u\in V$, we consider
\begin{equation}\label{eqn:subspace-S}
\parbox{.615\textwidth}{the subspace $\mathcal{S}$ of $\mathfrak{sp}(V,\Omega)$ corresponding under \eqref{eqn:isomorphism-spv} to the subspace of $S^2V^*$ spanned by elements of the form $\,\Omega \ovee (\underline{u}\odot \underline{z})(x,y,\cdot,\cdot)\,$, with $x,y,z\in V$.} 
\end{equation}
It is in the next result that simplicity of $\mathfrak{sp}(V,\Omega)$ plays a crucial role. Below, for each subspace $L$ of $V$, we set $L^\Omega = \{x\in V : \Omega(x,y) = 0\mbox{ for all }y\in L\}$.
\begin{prop}\label{prop:S-generates-spV}
  If $u\neq 0$ and $2n\geq 4$, the subalgebra $\mathfrak{u}$ generated by the subspace $\mathcal{S}$ in \eqref{eqn:subspace-S} in fact equals all of $\mathfrak{sp}(V,\Omega)$.
\end{prop}

\begin{proof}
  Our initial claim is that
  \begin{equation}\label{eqn:S-not-in-p}
    \parbox{.8\textwidth}{$\mathcal{S} \not\subseteq \mathfrak{p}$, where $\mathfrak{p} = \{A\in\mathfrak{sp}(V,\Omega) : Au\in \R u\}$ is the stabilizer of $\R u$.}
  \end{equation}
  As $A\in \mathfrak{p}$ if and only if $\Omega(Au,w) = 0$ for all $w\in (\R u)^\Omega$, we establish \eqref{eqn:S-not-in-p} by exhibiting $x,y\in V$ and $w\in (\R u)^\Omega$ such that $(\Omega \ovee (\underline{u}\odot\underline{x}))(x,y,u,w) \neq 0$. Namely, as $u\neq 0$ we may take $x\not\in (\R u)^\Omega$, while the assumption that $2n\geq 4$ allows us to take $y\in (\R u)^\Omega \smallsetminus (\R u)$ and $w\in (\R u)^\Omega$ with $\Omega(y,w) \neq 0$. With such choices, \eqref{eqn:SKN} immediately leads to $(\Omega \ovee (\underline{u}\odot\underline{x}))(x,y,u,w)= \Omega(x,u)^2\Omega(y,w) \neq 0$, as required.

  The next step is to show that
  \begin{equation}\label{eqn:u+p}
    \parbox{.59\textwidth}{$[\mathfrak{p},\mathfrak{u}] \subseteq \mathfrak{u}$, so that $\mathfrak{u}+\mathfrak{p}$ is a subalgebra of $\mathfrak{sp}(V,\Omega)$.}
  \end{equation}Let $A\in \mathfrak{p}$ and $B\in \mathcal{S}$ correspond under \eqref{eqn:isomorphism-spv} to $(\Omega \ovee (\underline{u} \odot \underline{z}))(x,y,\cdot,\cdot)$, for some $x,y,z\in V$; write also $Au=\lambda u$, with $\lambda \in \R$. Combining the fact that $A\cdot \underline{v} = \underline{Av}$ for all $v\in V$, relations \mbox{(\ref{eqn:basic-derivations}-i)} and \mbox{(\ref{eqn:basic-derivations}-iii)}, and Lemma \ref{lem:commutator[A,K]}, it follows that the commutator $[A,B]$ corresponds under \eqref{eqn:isomorphism-spv} to the sum of the four terms
  \begin{equation}
\begin{split} \lambda (\Omega \ovee( \underline{u}\odot \underline{z}))(x,y,\cdot,\cdot)&,\quad (\Omega \ovee (\underline{u} \odot \underline{Az}))(x,y,\cdot,\cdot), \\ \quad (\Omega\ovee (\underline{u} \odot \underline{z}))(Ax,y,\cdot,\cdot),\quad &\mbox{and}\quad (\Omega\ovee (\underline{u} \odot \underline{z}))(x,Ay,\cdot,\cdot),\end{split}
  \end{equation}
so that $[A,B]\in \mathcal{S}$ by \eqref{eqn:subspace-S}, and hence $[\mathfrak{p},\mathcal{S}]\subseteq \mathcal{S}$. The Jacobi identity guarantees that $\mathfrak{g} = \{B \in \mathfrak{u} : [\mathfrak{p},B]\subseteq \mathfrak{u}\}$ is a subalgebra of $\mathfrak{u}$, while at the same time the previous argument yields $\mathcal{S} \subseteq \mathfrak{g}$. Thus $\mathfrak{g}=\mathfrak{u}$, establishing \eqref{eqn:u+p}.

  Finally, observe that $\mathfrak{p}$---being the stabilizer of a line---is maximal as a subalgebra of $\mathfrak{sp}(V,\Omega)$ \cite[pp. 201--202, (III)]{Bourbaki_1975}. As $\mathfrak{u}+\mathfrak{p}$ is a subalgebra of $\mathfrak{sp}(V,\Omega)$ by \eqref{eqn:u+p}, and contains $\mathfrak{p}$ strictly by \eqref{eqn:S-not-in-p}, we necessarily have that $\mathfrak{u}+\mathfrak{p} = \mathfrak{sp}(V,\Omega)$. This last relation, in turn, immediately implies that $\mathfrak{u}$ is an ideal of $\mathfrak{sp}(V,\Omega)$. With $u\neq 0$ and $\mathfrak{sp}(V,\Omega)$ being simple, cf. \cite[Theorem 6.105]{Knapp_2002}, it follows that $\mathfrak{u} = \mathfrak{sp}(V,\Omega)$ as required.
\end{proof}

\begin{obs}\label{obs:p-maximal-Riem}
  Maximality of $\mathfrak{p}$ in \eqref{eqn:S-not-in-p} also holds in the Riemannian setting. Namely, if for a pseudo-Euclidean vector space $(V,g)$ and $v\in V\smallsetminus \{0\}$ we set $\mathfrak{p} = \{A\in \mathfrak{so}(V,g) : Av=0\}$, then $\mathfrak{p}$ is guaranteed to be maximal as a subalgebra of $\mathfrak{so}(V,g)$ whenever $g$ is positive-definite. Indeed, we first note that the mapping $\varphi\colon v^\perp \to \mathfrak{so}(V,g)/\mathfrak{p}$ given by $\varphi(w) = v\wedge w + \mathfrak{p}$, easily seen to be $\mathfrak{p}$-equivariant, is an isomorphism if and only if $g(v,v) \neq 0$: it has $\ker \!(\varphi) = \{w\in v^\perp : g(v,v)w=0\}$ and $\varphi(g(v,v)^{-1}Av) = A+\mathfrak{p}$ for all $A\in \mathfrak{so}(V,g)$, whenever $g(v,v)\neq 0$. We then claim that $v^\perp$ is irreducible as a $\mathfrak{p}$-module whenever $g$ is positive-definite, thus implying that $\mathfrak{p}$ is maximal. To establish this, we show that if a nontrivial subspace $Z \subseteq v^\perp$ is $\mathfrak{p}$-invariant and not null, then $Z=v^\perp$: fix $z\in Z\smallsetminus \{0\}$ such that $g(z,z)\neq 0$ and observe that, as $w = g(z,z)^{-1} (z\wedge w)z \in Z$ whenever $w\in z^\perp \cap v^\perp$ (using that $z\wedge w \in \mathfrak{p}$), we obtain $z^\perp \cap v^\perp \subseteq Z$; the conclusion follows since, decomposing any $x\in v^\perp$ as $x = \lambda z+x_0$ with $\lambda \in \R$ and $x_0 \in z^\perp\cap v^\perp$, both $\lambda z$ and $x_0$ lie in $Z$.
\end{obs}

\section{Differential identities involving the Ricci tensor}\label{sec:proof-teoA}

Hereafter, let $(M,\omega,\nabla)$ be a $2n$-dimensional Fedosov manifold, $R$ its curvature tensor, and $W = W(R)$ be its symplectic Weyl tensor as in \eqref{eqn:defn_Weyl}. As in Section \ref{sec:symp-curvature}, for any vector field $X$ on $M$ we let $\underline{X} = \omega(X,\cdot)$ be the corresponding $1$-form.

\medskip

Recall that, in addition to the symmetries \eqref{eqn:defn-symp-curv}, $R$ also satisfies the differential Bianchi identity ${\rm d}^\nabla R=0$, where ${\rm d}^\nabla$ is the covariant exterior derivative operator and $R$ is regarded as a \mbox{$2$-form} valued in ${\rm End}(TM)$. As $\nabla$ is torsionfree and parallelizes $\omega$, this identity may be rewritten in terms of the fully-covariant curvature tensor---with the convention that $R(X,Y,Z,S) = \omega(R(X,Y)Z,S)$---as
\begin{equation}\label{eqn:diff-Bianchi}
  (\nabla_XR)(Y,Z,S,T) + (\nabla_YR)(Z,X,S,T) + (\nabla_ZR)(X,Y,S,T) = 0.
\end{equation}
Furthermore, the covariant exterior derivative of ${\rm Ric}$ is given by
\begin{equation}\label{eqn:dRic}
  ({\rm d}^\nabla{\rm Ric})(X,Y)Z = (\nabla_X{\rm Ric})(Y,Z) - (\nabla_Y{\rm Ric})(X,Z),
\end{equation}if ${\rm Ric}$ is regarded as a $1$-form valued in $T^*M$.

\medskip

In \cite[Lemma 1(i)]{CGHR_2001} it was shown that, whenever $W=0$, there is a vector field $U$ on $M$ such that
\begin{equation}\label{eqn:U-field}
\nabla_Z{\rm Ric} = -\frac{1}{2n+1}\,\underline{U}\odot \underline{Z}  ,\quad\mbox{for all }Z\in\mathfrak{X}(M). 
\end{equation}Up to a constant, $U$ is dual to the divergence $1$-form $({\rm div}_\omega{\rm Ric})_k = \omega^{ij}\nabla_iR_{jk}$, which plays the role of ``differential of the scalar curvature,'' cf. \cite[Section 3]{Fox_2017}. A major step towards the proof of Theorem \ref{teo:symplectic-Roter} is to replace $W=0$ with the weaker condition $\nabla W=0$. In order to do so, for each $Z\in \mathfrak{X}(M)$ we consider the differential forms
\begin{equation}\label{eqn:phi-and-sigma}
  \parbox{.54\textwidth}{$\phi_Z \in \Omega^1(M)$ and $\sigma_Z\in \Omega^2(M)$, given by $\phi_Z = (\nabla{\rm Ric})(Z,Z)$ and $\sigma_Z = ({\rm d}^\nabla {\rm Ric})(\cdot,\cdot)Z$,}
\end{equation}
and explore relations between them.

\begin{lem}\label{lem:Bianchi-3-forms}
  If $\nabla W = 0$, then $\omega\wedge \phi_Z + \underline{Z} \wedge \sigma_Z = 0$ for all $Z\in \mathfrak{X}(M)$.
\end{lem}

\begin{proof}
  Under the condition $\nabla W = 0$, as covariant derivatives naturally satisfy a Leibniz rule for the symplectic Kulkarni-Nomizu product, it follows from  \eqref{eqn:defn_Weyl} that \eqref{eqn:diff-Bianchi} reduces to
  \begin{equation}\label{eqn:Bianchi-reduced}
    \begin{split}
          (\omega \ovee \nabla_X{\rm Ric})(Y,Z,S,T) + &    (\omega \ovee \nabla_Y{\rm Ric})(Z,X,S,T) \\ &+    (\omega \ovee \nabla_Z{\rm Ric})(X,Y,S,T) =0.
    \end{split}
  \end{equation}
 By using \eqref{eqn:SKN}, grouping terms, and recognizing the definition of $\phi_Z$ in \eqref{eqn:phi-and-sigma}, we have that
  \begin{equation}\label{eqn:SKN-repeated}
    \begin{split}
          \frac{1}{2} (\omega \ovee \nabla_X{\rm Ric})(Y,T,Z,Z) &=  (\nabla_X{\rm Ric})(T,Z)\omega(Y,Z) \\ &\quad- (\nabla_X{\rm Ric})(Y,Z)\omega(T,Z)+\omega(Y,T) \phi_Z(X).
    \end{split}
  \end{equation}We now cyclically sum \eqref{eqn:SKN-repeated} over $X,Y,T$. The left side of \eqref{eqn:SKN-repeated} sums to zero by \eqref{eqn:Bianchi-reduced}, while the last term in the right side sums to $(\omega\wedge \phi_Z)(X,Y,T)$. The remaining six terms in the cyclic sum coming from the last two terms in \eqref{eqn:SKN-repeated} unaccounted for are grouped in pairs with the aid of \eqref{eqn:dRic} and, recognizing the definition of $\sigma_Z$ in \eqref{eqn:phi-and-sigma}, amount to $(\underline{Z}\wedge \sigma_Z)(X,Y,T)$ as required.
\end{proof}

To conclude this section, consider the Lefschetz and dual-Lefschetz operators $L\colon \Omega^k(M)\to \Omega^{k+2}(M)$ and $\Lambda\colon \Omega^k(M)\to \Omega^{k-2}(M)$, defined by
\begin{equation}\label{eqn:Lefschetz}
  L\alpha = \omega\wedge \alpha\quad\mbox{and}\quad \Lambda\alpha = \frac{1}{2}{\rm tr}_\omega((X,Y)\mapsto \alpha(X,Y,\cdot,\ldots,\cdot))
\end{equation}for $\alpha \in \Omega^k(M)$; see e.g. \cite[Section 3.2]{Shane_DGA}. A straightforward calculation (e.g., using Darboux coordinates) shows that
\begin{equation}\label{eqn:identities_Lambda}
  \Lambda(\omega \wedge \alpha) = (n-1)\alpha\quad\mbox{and}\quad \Lambda( \underline{X}\wedge \beta) = (\Lambda \beta) \underline{X} - \beta(X,\cdot)
\end{equation}whenever $\alpha\in \Omega^1(M)$ and $\beta\in\Omega^2(M)$.

\begin{prop}\label{prop:DRic-U}
  If $2n\geq 4$ and $\nabla W = 0$, there is a vector field $U$ on $M$ for which \eqref{eqn:U-field} holds. In particular, $(M,\omega,\nabla)$ is locally symmetric if and only if $U=0$.
\end{prop}

\begin{proof}
  The vector field $U$ is defined by requiring that $\omega(U,Z) = - \Lambda \sigma_Z$ for all $Z\in \mathfrak{X}(M)$, where $\Lambda$ and $\sigma_Z$ as in \eqref{eqn:Lefschetz} and \eqref{eqn:phi-and-sigma}; note that the dependence of $\Lambda\sigma_Z$ on $Z$ is indeed tensorial, making $U$ well-defined. Combining the definitions of $\phi_Z$ and $\sigma_Z$ with \eqref{eqn:dRic}, we see that
  \begin{equation}\label{eqn:sigma-versus-phi}
    \sigma_Z(Z,X) = (\nabla_Z{\rm Ric})(X,Z) - \phi_Z(X)
  \end{equation}for all $X,Z\in \mathfrak{X}(M)$. At the same time, applying $\Lambda$ to the identity obtained in Lemma \ref{lem:Bianchi-3-forms} and using \eqref{eqn:identities_Lambda}, we see that $(n-1)\phi_Z + (\Lambda \sigma_Z)\underline{Z}-\sigma_Z(Z,\cdot)=0$. Substituting \eqref{eqn:sigma-versus-phi} into this last relation and using the definition of $U$, it follows that $n (\nabla_X{\rm Ric})(Z,Z) = \omega(U,Z)\omega(Z,X)+(\nabla_Z{\rm Ric})(X,Z)$. Polarization in the variable $Z$ now yields
  \begin{equation}\label{eqn:2n-first}
    2n(\nabla_X{\rm Ric})(Y,Z) =  - (\underline{U} \odot \underline{X})(Y,Z) + (\nabla_Z{\rm Ric})(X,Y)+ (\nabla_Y{\rm Ric})(X,Z).
  \end{equation}Switching the roles of $X$ and $Y$ in \eqref{eqn:2n-first} and adding the resulting identity to \eqref{eqn:2n-first} itself, we obtain
  \begin{equation}\label{eqn:2n-second}
    (2n-1)(\nabla_X{\rm Ric}(Y,Z)+(\nabla_Y{\rm Ric})(X,Z)) = 2 (\nabla_Z{\rm Ric})(X,Y) +  (\underline{U}\odot \underline{Z})(X,Y).
  \end{equation}This time switching the roles of $X$ and $Z$ in \eqref{eqn:2n-first}, and substituting it into \eqref{eqn:2n-second}, it follows that
  \begin{equation}\label{eqn:2n-third}
    (2n-1)(2n \nabla_Z{\rm Ric} + \underline{U}\odot \underline{Z}) = 2\nabla_Z{\rm Ric} + \underline{U}\odot \underline{Z},
  \end{equation}as symmetric tensor fields. Dividing both sides of \eqref{eqn:2n-third} by $2-2n\neq 0$ and solving for $\nabla_Z{\rm Ric}$, we finally arrive at \eqref{eqn:U-field}.
\end{proof}

\begin{obs}
  Both Lemma \ref{lem:Bianchi-3-forms} and Proposition \ref{prop:DRic-U} (but not necessarily Theorem \ref{teo:symplectic-Roter} itself) hold under the more general assumption that $W$ satisfies \eqref{eqn:diff-Bianchi}.
\end{obs}

\section{Proof of Theorem \ref{teo:symplectic-Roter}}\label{sec:proof-teoA}

We continue with the setup from the previous section, and assume from here onwards that $2n\geq 4$ and $\nabla W=0$. The latter implies that $R(X,Y)\cdot W = 0$ for all $X,Y$, cf. \eqref{eqn:defn-derivation}, and further differentiation leads to $(\nabla_ZR)(X,Y)\cdot W = 0$ for all $X,Y,Z$. Taking into account \eqref{eqn:defn_Weyl} and Proposition \ref{prop:DRic-U}, we obtain
\begin{equation}\label{eqn:S-annihilates-W}
(\omega \ovee \big(\underline{U}\odot \underline{Z})\big)(X,Y)\cdot W = 0,\qquad\mbox{for all }X,Y,Z\in \mathfrak{X}(M).
\end{equation}
If $\nabla R \neq 0$, and hence $\nabla {\rm Ric} \neq 0$, Proposition \ref{prop:DRic-U} yields a point $p\in M$ such that $U_p \neq 0$. We then set $(V,\Omega) = (T_pM, \omega_p)$ and consider the subspace $\mathcal{S}$ of $\mathfrak{sp}(V,\Omega)$ defined via \eqref{eqn:subspace-S} for $u=U_p$. It is clear from \eqref{eqn:S-annihilates-W} that $\mathcal{S} \cdot W_p = 0$, so that \mbox{(\ref{eqn:basic-derivations}-ii)} and Proposition \ref{prop:S-generates-spV} together lead to $\mathfrak{sp}(V,\Omega) \cdot W_p = 0$. Proposition \ref{prop:killed-by-derivations} then implies that $W_p=0$ and, as $\nabla W=0$ and $M$ is connected, it finally follows that $W=0$. $\hfill\square$

\bibliography{symplectic_refs}{}
\bibliographystyle{plain}

\end{document}